\documentclass{amsart}

\usepackage{url}
\usepackage{graphics}
\usepackage{a4wide}
\usepackage{amsmath, amsfonts}
\usepackage{epstopdf}
\usepackage{graphicx}
\usepackage{color}

\newtheorem{theorem}{Theorem}
\newtheorem{lemma}[theorem]{Lemma}
\newtheorem{corollary}[theorem]{Corollary}

\newtheorem{observation}[theorem]{Observation}

\newcommand{\T}{{\mathcal T}}

\title{Sharp upper and lower bounds on a restricted class of convex characters}

\author{Steven Kelk}

\author{Ruben Meuwese}

\address{Department of Data Science and Knowledge Engineering (DKE), Maastricht University, P.O. Box 616, 6200 MD Maastricht, The Netherlands}

\email{steven.kelk@maastrichtuniversity.nl}
\email{ruben.meuwese@maastrichtuniversity.nl}

\begin{document}

\begin{abstract}
Let $\T$ be an unrooted binary tree with $n$ distinctly labelled leaves. Deriving its name from the field of phylogenetics, a \emph{convex character} on $\T$ is simply a partition of the leaves such that the minimal spanning subtrees induced by the blocks of the partition are mutually disjoint. 
In earlier work Kelk and Stamoulis (\emph{Advances in Applied Mathematics} 84 (2017), pp. 34--46)
defined $g_k(\T)$ as the number of convex characters where each block has at least $k$ leaves. Exact expressions were given for $g_1$ and $g_2$, where the topology of $\T$ turns out to be irrelevant, and it was noted that for $k \geq 3$ topological neutrality no longer holds. In this article, for every $k \geq 3$ we describe tree topologies achieving the maximum and minimum values of $g_k$ and determine corresponding expressions and exponential bounds for $g_k$. Finally, we reflect briefly on possible algorithmic applications of these results.

\end{abstract}

\maketitle

\section{Introduction}
\label{sec:intro}

Consider an unrooted, undirected binary tree $\T$ where the $n$ leaves are bijectively labelled by a set $X$ of labels. In this article we are interested in  partitions of $X$ whereby the minimal spanning trees induced by the blocks, are disjoint in $\T$. In the field of mathematical phylogenetics, where $\T$ represents a hypothesis about the evolutionary history of a set of contemporary species $X$ and interior nodes represent hypothetical ancestors, such partitions are called \emph{convex characters} \cite{semple_steel_2003}. Convex characters are of interest because they represent a `most parsimonious' evolutionary scenario. Specifically, a block of the partition can be viewed as a subset of $X$ that shares a certain trait e.g. has a backbone. The fact that the induced spanning trees do not overlap, models the situation whereby these traits can be extended to ancestors in $\T$, such that all contemporary and ancestral species that share the trait form a connected component. In other words: the trait emerges once in history; it does not vanish and re-emerge multiple times. Convex characters, which are sometimes called \emph{homoplasy-free} characters due to the absence of recurrent mutation, 
play a central role in 
the well-studied \emph{perfect phylogeny} problem
\cite{kannan1997fast} and a number of other combinatorial and algorithmic problems inspired by phylogenetics (e.g. \cite{moran2007efficient, fischer2014maximum, alexeev2018combinatorial}).
We refer to standard texts such as \cite{steel2016phylogeny,felsenstein2004inferring} for more background on (mathematical) phylogenetics.

In \cite{kelk2017note} the question was posed: how many convex characters are there on $\T$ where each block has at least $k$ leaves, denoted $g_k(\T)$? The authors proved that for $k \in \{1,2\}$ the answer is a Fibonacci number, and  independent of the topology of $\T$, but that for $k=3$ the topology of the tree does matter. This raised the question of establishing lower and upper bounds on $g_k$, for $k \geq 3$.

In this article we give, for every $k\geq 3$, sharp upper and lower bounds. We do this by identifying tree topologies that provably obtain the maximum and minimum and determining $g_k$ on these trees. The maximum is attained by caterpillar trees, and the minimum by trees we call \emph{fully $k$-loaded trees}. Informally, fully $k$-loaded trees are trees in which as often as possible the leaves are organized in size $k-1$ clusters at the periphery of the tree. For the lower bound, an exact expression is given: it grows at $\Theta( \phi^{\frac{n}{k-1}})$ where $\phi \approx 1.618$ is the golden ratio. For the upper bound, an exact expression and an exponential rate of growth can be obtained by determining the real positive root of a characteristic polynomial induced by a \textcolor{black}{homogeneous linear recurrence}, and then solving for initial conditions.

In Section \ref{sec:prelim} we establish preliminaries. In Section \ref{sec:cat} we prove the upper bound. In Section \ref{sec:full} we prove the lower bound and pause briefly to study the behaviour of $g_3$ between its upper and lower bounds. In Section \ref{sec:app} we reflect on potential algorithmic applications of these results, noting that algorithms listing all $g_k$ characters speed up considerably as $k$ increases. Finally, in Section \ref{sec:future} we list a number of open problems.



\section{Preliminaries}\label{sec:prelim}

We note that the results in the article are not specific to phylogenetics: they apply to undirected binary trees with distinctly labelled leaves. However, to ensure consistency with the phylogenetics literature that inspired the research we adopt standard phylogenetic notation.

 An {\it unrooted binary phylogenetic $X$-tree} is an undirected, unrooted tree $\T =(V(\T),E(\T))$ where every internal vertex has degree 3 and whose leaves are bijectively labelled by a set $X$, where $X$ is
often called the set of \emph{taxa} (representing the contemporary species). All the trees in this articles are unrooted binary phylogenetic trees, so we simply write \emph{tree} for brevity.  
We let $n=|X|$. For $X' \subseteq X$ we write $\T[X']$ to denote the unique minimal subtree of $\T$ that spans $X'$, and $\T|X'$ to denote the phylogenetic tree obtained from $\T[X']$ by repeatedly suppressing nodes of degree 2. For $X' \subset X$ we define $\T \setminus X'$ to be $\T|(X \setminus X')$.

A \emph{character} $f$ on $X$ is simply a partition of $X$ into blocks (i.e. \emph{non-empty} subsets) $X_1, X_2, ..., X_m$. Following the phylogenetics literature we will often refer to the blocks of the partition as \emph{states}. We write $X_1|X_2|...|X_m$ to denote that the subsets $X_1, X_2, ..., X_m$ form a  partition/character of $X$.
A character $X_1|X_2|...|X_m$ is said to be \emph{convex} on $\T$ if, for each $i \neq j$,  $T[X_i]$ is disjoint from $T[X_j]$.  
The convexity of a character on a tree $\T$ can be tested in polynomial \cite{fitch_1971,hartigan1973minimum} (in fact, linear \cite{bachoore2006convex}) time. \textcolor{black}{For the tree $\T$  shown in Figure \ref{fig:g2}, $abde|c|fg$ is an example of a convex character.}

\begin{figure}[h]
\centering
\includegraphics[scale=0.2]{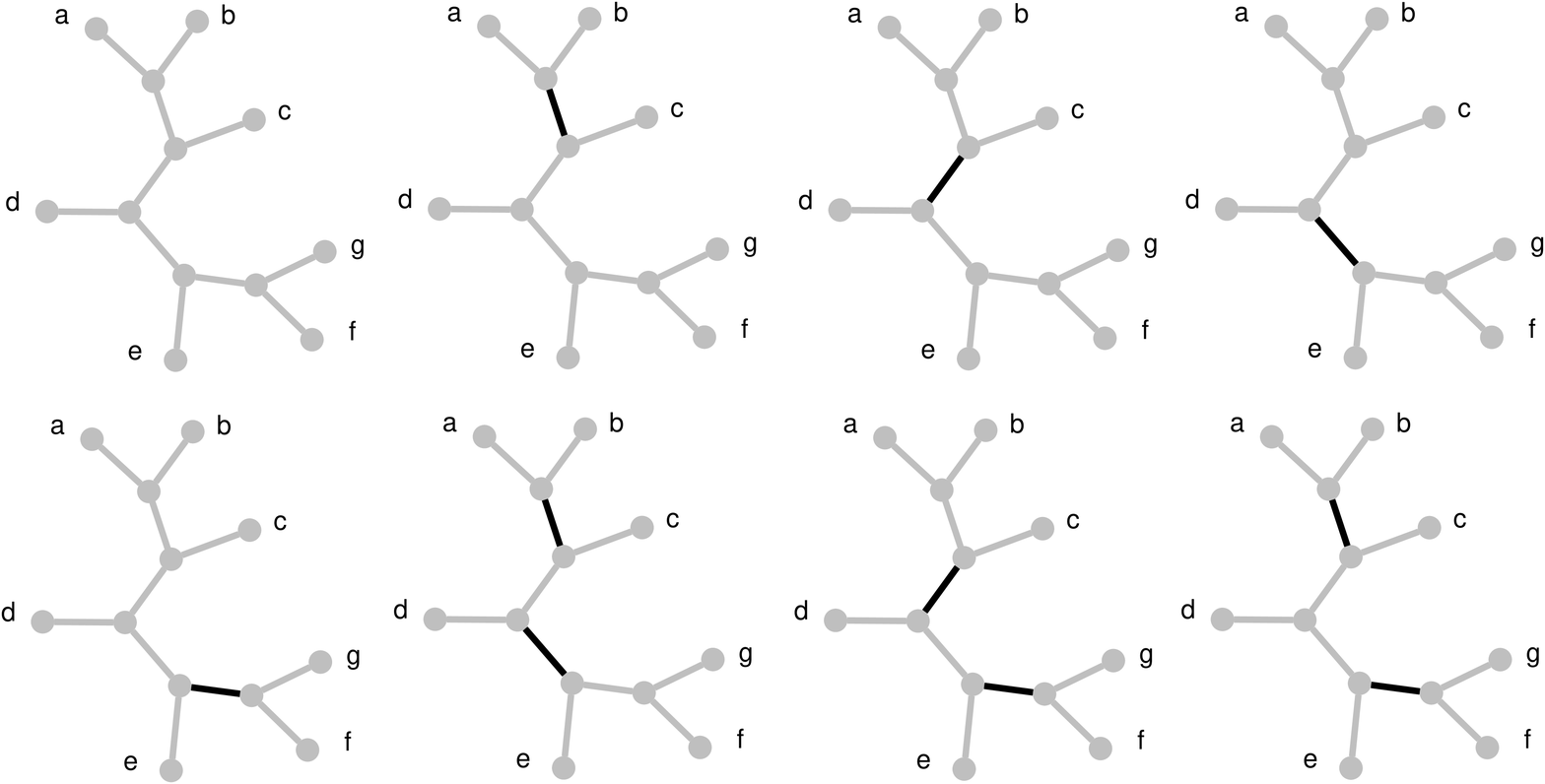}
\caption{For the given tree $\T$ (on 7 taxa) there are 233 convex characters in total,
but only 8 in which each state (i.e. block of the partition) contains at least 2 taxa, and these are shown above. Taxa connected by grey edges are in  the same state of the character. The 3 $g_3$ characters are $abcdefg$, $abc|defg$ and $abcd|efg$, and the single $g_4$ character is $abcdefg$. 
}
\label{fig:g2} 
\end{figure}

Following \cite{kelk2017note}, we define $g_k(\T)$ to be the number of convex characters on $\T$ where each state contains at least $k$ taxa. For shorthand we refer to such characters as $g_k$ characters; note that $g_1$ characters are just convex characters without any further restrictions. See Figure \ref{fig:g2} for an example. In \cite{kelk2017note} it was proven that $g_1(\T)$ and $g_2(\T)$ are equal to the $(2n-1)$th and $(n-1)$th Fibonacci number, respectively. This yielded the following exact expressions, where $\T$ is replaced by $n$ due to the topological neutrality, and $\phi \approx 1.618$ is the golden ratio.
\begin{align*}
g_1(n) &= \bigg \lfloor \frac{\phi^{2n-1}}{\sqrt{5}} + \frac{1}{2} \bigg \rfloor,\\
g_2(n) &= \bigg \lfloor \frac{\phi^{n-1}}{\sqrt{5}} + \frac{1}{2} \bigg \rfloor.
\end{align*}

We will use the following simple observation repeatedly; it shows that for very small $n$ (relative to $k$) $g_k$ is independent of topology.

\begin{observation}
\label{obs:gettingstarted}
Let $\T$ be a tree on $n$ taxa. If $n<k$, $g_k(\T)=0$, and if $k \leq n < 2k$, $g_k(\T)=1$.
\end{observation}
\begin{proof}
Every character contains at least one state, and each state must have at least $k$ taxa, so for $n<k$ no $g_k$ characters can exist. If $k \leq n$ we have $g_k(\T) \geq 1$ because we can always take the unique convex character with a single size-$n$ state. If \textcolor{black}{$n < 2k$} then it is not possible to have 2 or more $g_k$ characters, because at least one of them would need to have 2 or more states, and each such state must contain at least $k$ taxa.
\end{proof}

Let $A|B$ be a bipartition of $X$.
We say that a tree $\T$ on $X$ contains the \emph{split} $A|B$ if there is a \textcolor{black}{single} edge of $\T$ whose deletion disconnects $\T$ into two components, where $A$ is the set of taxa in one component and $B$ is the set of taxa in the other. The next observation will also be used repeatedly.

\begin{observation}
\label{obs:atleasttwo}
Let $\T$ be a tree on $n$ vertices and suppose $\T$ contains a split $A|B$ such that $|B| \leq k$. Then every $g_k$ character of $\T$ includes a state that is a superset of $B$. If $|B| < k$, then every $g_k$ character of $\T$ includes a state that is a \underline{strict} superset of $B$.
\end{observation}
\begin{proof}
Firstly, recall that every taxon in $B$ has to appear in some state. Suppose for the sake of contradiction that there is a $g_k$ character such that $B$ intersects with two or more states in the character. Given that $|B| \leq k$, each of these states contains at most $k-1$ taxa from $B$ (because each of the states must include at least one taxon from $B$, and states are disjoint). Hence, each of these states also contains at least one taxon from $A$. But then the subtrees induced by these two (or more) states are not disjoint; in particular, they both use the edge corresponding to split $A|B$, contradicting the convexity of the character. Hence, $B$ intersects with at most one state. If $|B|=k$, then this state must necessarily include all of $B$. If $|B|<k$ then this state must also include at least one taxon from $A$, and we are done.
\end{proof}


The next lemma establishes a useful recurrence.

\begin{lemma}
\label{lem:decrease}
Suppose $\T$ contains a split $A|B$ where $|A| = k$, and let $x$ be an arbitrary element of $A$. Then $g_k(\T) = g_k(\T \setminus A) + g_k(\T \setminus \{x\})$.
\end{lemma}
\begin{proof}
Consider an arbitrary $g_k$ character $f$ of $\T$. Due to Observation \ref{obs:atleasttwo} and the fact that $|A|=k$, there are only two cases. In the first case, $f$ contains a state \emph{equal} to $A$. There are exactly $g_k(\T \setminus A)$ such characters.

In the second case, $f$ contains a state $X_i$ that is a strict superset of $A$, and which thus also intersects with $B$; the state contains at least $k+1$ taxa. There are $g_k(\T \setminus \{x\})$ such characters. To see this, note that removing $x$ from $X_i$ gives a state that still has at least $k$ taxa, so this yields a $g_k$ character for $\T \setminus \{x\}$. In the other direction, a $g_k$ character for $\T \setminus \{x\}$ necessarily includes a state that contains all of $A \setminus \{x\}$ and at least one taxon from $B$. Adding $x$ to this state yields a state that contains all of $A$, and at least one taxon of $B$.
\end{proof}

\section{Caterpillars maximize $g_k(\T)$ for every $k \geq 1$}
\label{sec:cat}

\begin{figure}[h]
\centering
\includegraphics[scale=0.18]{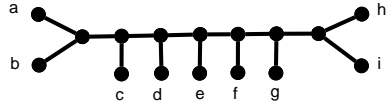}
\caption{A caterpillar tree on 9 taxa.}
\label{fig:cat} 
\end{figure}

We say that two distinct taxa $x,y$ form
a \emph{cherry} of a tree $\T$ if they have a common parent. For $n \geq 4$ a \emph{caterpillar} is a tree on $n$ taxa which
has exactly two cherries. Equivalently, a caterpillar is a tree where all the degree-3
nodes form a path. See Figure \ref{fig:cat} for an example. For convenience we also
regard the unique
trees on 1, 2 or 3 taxa to be caterpillars. We write $Cat_n$ to denote the caterpillar on $n \geq 1$ vertices.

We start by describing a recurrence for $g_k(Cat_n)$ and then prove that this is the maximum value of $g_k(\T)$ ranging over all trees $\T$ with $n$ taxa.

Observation \ref{obs:gettingstarted} establishes the initial values of $g_k(Cat_n)$. 

\begin{lemma}
\label{lem:catrecurse}
For $n > k \geq 2$, $g_k(Cat_n) = g_k(Cat_{n-1}) + g_k(Cat_{n-k})$. As a consequence, $g_k(Cat_n) \in \Theta(\alpha^n)$ where $\alpha$ is the positive real root of the polynomial $x^{k}-x^{k-1}-1$.
\end{lemma}
\begin{proof}
Due to its regular topology $Cat_n$ definitely contains a split $A|B$ where $|A|=k$. Observe that $\T \setminus A$ is a caterpillar, and $\T \setminus \{x\}$ (for any $x \in X$) is also a caterpillar. The recurrence now follows directly from Lemma \ref{lem:decrease}. \textcolor{black}{This is a homogeneous linear recurrence; it is well understood how to solve such recurrences (see standard textbooks such as \cite{knuth1989concrete}).}  In particular, it follows that $g_k(Cat_n)$ grows at rate $\Theta(\alpha^n)$ where $\alpha$ is the positive real root of the characteristic polynomial $x^k - x^{k-1}-1$. A precise expression for $g_k(Cat_n)$ can then be derived, if desired, by taking the initial terms of the recurrence into account. \end{proof}

\noindent
We give an explicit example for $g_3$. The closed expression is:
\[
g_3(Cat_n) = \bigg \lfloor 0.194225... \cdot 1.46557...^n + \frac{1}{2}\bigg \rfloor
\]
where 1.46557... is the real solution to $x^3 - x^2 -1 = 0$. The value 0.194225... is obtained by taking the real solution of $31x^3 - 31x^2 + 9x - 1=0$ and dividing it by $(1.46557...)^3$ to adjust for the fact that in this expression $n$ refers to the number of taxa. See: \url{http://oeis.org/A000930}\footnote{This recurrence is also known as \emph{Narayana's cows sequence}.}.\\
\\
Let $A|B|C$ be a tripartition of $X$.
We say that a tree $\T$ contains $A|B|C$ if there is a degree-3 node $u$ such that $A, B, C$ are the subsets of taxa of the three subtrees incident at $u$. 

Suppose $\T$ contains a tripartition $A|B|C$ where $|A|, |B|, |C| \geq 2$. If we delete the $C$ subtree and replace it with a length-$|C|$ caterpillar inserted between the $A$ and $B$ subtrees, we obtain a new tree $\T'$ and say that this is the result of \emph{linearizing} the $C$ subtree. See Figure \ref{fig:linearize} for an example.

\begin{figure}[h]
\centering
\includegraphics[scale=0.20]{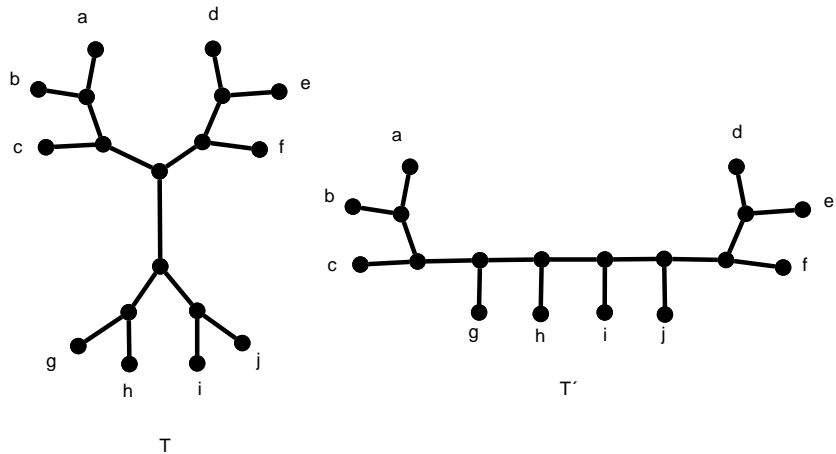}
\caption{An example of the linearization operation, acting on the tripartition $A|B|C$ of $\T$ where $A=\{a,b,c\}$, $B = \{d,e,f\}$ and $C=\{g,h,i,j\}$, to yield $\T'$. Linearization replaces subtree $C$  by a caterpillar.}
\label{fig:linearize} 
\end{figure}

\begin{lemma}
\label{lem:linearize}
Let $k \geq 3$. Suppose $\T$ contains a tripartition $A|B|C$ where $
|A|,|B|,|C| \geq 2$ and $|C|<k$. Let $\T'$ be the result of linearizing the $C$ subtree. Then $g_k(\T') \geq g_k(\T)$ and $\T'$ has fewer cherries than $\T$.
\end{lemma}
\begin{proof}
The crucial fact here is that $|C|<k$. Specifically, by Observation \ref{obs:atleasttwo}, a $g_k$ character $f$ for $\T$ must
contain a state $X_i$ that includes all of $C$ and
which additionally intersects with at least one of $A$ and $B$. As a result, there is no state $X_j \neq X_i$ in $f$ that intersects both $A$ and $B$, but not $C$ - because then $\T[X_j]$ and $\T[X_i]$ would intersect. Hence, $f$ is also a $g_k$ character of $\T'$.

The $C$ subtree contained at least one cherry, because $|C| \geq 2$; these are all destroyed by linearization. However, no new cherries are created because $|A|,|B| \geq 2$.
\end{proof}

\begin{theorem}
Let $k \geq 1$. $g_k(\T) \leq g_k(Cat_n)$ for every tree $\T$ on $n$ taxa. 
\end{theorem}
\begin{proof}
For $k \in \{1,2\}$ the result follows automatically because $g_1$ and $g_2$ are invariant for topology. Henceforth we assume $k \geq 3$.

We use induction on $n+t$, where $t$ is the number of cherries in $\T$. For the base case, recall from Observation \ref{obs:gettingstarted} that for $n < k$
we have $g_k(\T) = 0$,
and for $k \leq n < 2k$ we have $g_k(\T) = 1$, irrespective of the topology of $\T$, so the claim holds vacuously when $n<2k$. Now, observe that if $n+t \leq 2k$ and $n\geq 4$, then $n<2k$ (because $t \geq 2$ in any tree with at least 4 taxa) - so, again, \textcolor{black}{by topological neutrality} the claim holds. For $n\leq 3$, the definition of $t$ is somewhat ambiguous, but however one defines it we have $t \leq 3$, so $n+t \leq 6 \leq 2k$, and the claim clearly holds for such \textcolor{black}{very small} trees. Summarizing, we have shown that the claim always holds for $n+t \leq 2k$, which concludes the base case.

Next, the inductive step. Assume that the claim holds for all trees where $n+t \leq N$ (where $N \geq 2k$). We show that the claim holds for all trees where $n+t \leq N+1$. Consider, therefore, a tree $\T$ where $n+t$ is equal to $N+1$. (If no such tree exists, we are immediately done). If $n < 2k$, then we are immediately done by the earlier argument. So assume that \textcolor{black}{$n \geq 2k$}.

Suppose $\T$ contains a tripartition $A|B|C$ tripartition where $|A|,|B|,|C| \geq 2$ and $|C|<k$. Then by linearizing $\T$ to obtain $\T'$, we obtain via Lemma \ref{lem:linearize} that $g_k(\T') \geq g_k(\T)$ and that $\T'$ has fewer cherries than $\T$, but the same number of taxa, so by induction we have $g_k(Cat_n) \geq g_k(\T') \geq g_k(\T)$ and we are done. 

Let us therefore assume that the tripartition from the previous paragraph does not exist. Instead, consider any tripartition $A|B|C$ such that $|A|,|B| \geq 1$, $|C| \leq k$ and $|C|$ is maximized. (Such a tripartition must exist, because $\T$ contains a cherry, so taking $|C|=2$ is always possible in the worst case.) We claim that $|C|=k$. Suppose this was not so, i.e. that $2\leq |C| <k$. Then exactly one of $|A|$ and $|B|$ would be equal to 1. If neither was equal to 1, we would be in the earlier linearization case, and if they would both be equal to 1, then $n = |A|+|B|+|C| = |C|+2 \leq k+1 < 2k$ which is not possible due to the earlier assumption that \textcolor{black}{$n \geq 2k$}. So assume without loss of generality that $|B|=1$ and $|A| \geq 2$. It follows that $\T$ contains a new tripartition $A'|B'|C'$ where $|C'| = |C|+1$, $|A'|, |B'| \geq 1$ and $|A'| + |B'| = |A|$. \textcolor{black}{However, this new tripartition yields a contradiction to the assumption that $|C|$ 
was maximum: $A'|B'|C'$ also satisfies the required conditions but has $|C|'  > |C|$ (and $|C|' \leq k$ because $|C| < k$). So, indeed, $|C|=k$.}

Now, it follows that $\T$ contains the split $C|(A \cup B )$.  
Hence, Lemma \ref{lem:decrease} can be applied to $\T$. This gives:
\[
g_k(\T) = g_k(\T \setminus C) + g_k(\T \setminus \{x\})
\]
where $x$ is an arbitrary taxon in $C$. Tree $\T \setminus C$ has $n-k$ taxa and less than or equal to the same number of cherries as $\T$ (because at most one new cherry can be created by pruning the $C$ subtree, but at least one is destroyed, since $|C| = k \geq 2$). Tree $\T \setminus \{ x \} $ has $n-1$ taxa and less than or equal to the same number of cherries as $\T$. So the inductive claim holds for both these two trees. So $g_k(\T \setminus C) \leq g_k(Cat_{n-k})$ and $g_k(\T \setminus \{x\}) \leq g_k(Cat_{n-1})$. Hence, by Lemma \ref{lem:catrecurse} we have that $g_k(\T) \leq g_k(Cat_n)$, and we are done.
 \end{proof}

\begin{corollary}
For every $n$, the maximum value of $g_k$ ranging over all trees on $n$ taxa is $g_k(Cat_n)$, which is $\Theta(\alpha^n)$, where $\alpha$ is the positive real root of the characteristic polynomial $x^k - x^{k-1}-1$.
\end{corollary}

\section{Fully $k$-loaded trees minimize $g_k$ for every $k$}
\label{sec:full}

Having established that caterpillars maximize $g_k$, we now turn to the question of the minimum. We begin with a number of auxiliary lemmas and observations.

\begin{lemma}
\label{lem:nicecut}
Let $\T$ be a tree on $n > k$ taxa. Then $\T$ contains a split $A|B$ such that $k \leq |B| \leq 2(k-1)$.
\end{lemma}
\begin{proof}
Pick an arbitrary taxon of $\T$ and orient the edges of $\T$ away from this taxon. We label each edge with the number of taxa reachable from the head of this edge by directed paths (i.e. the number of taxa on the ``far'' side of the split induced by this edge). Starting from the edge incident to our chosen taxon, which has label $n-1 \geq k$, we walk onto an outgoing edge if it is labelled by a number $\geq k$ (if there are two such edges we can break ties arbitrarily) and we continue this walk until it cannot be extended any further. Let $e$ be the edge we have reached. At this point both outgoing arcs from $e$ are labelled by at most $k-1$; if the label was higher we could have continued the walk. Hence, the label of $e$ is at most $2(k-1)$. Also, the sum of the labels on these two edges is at least $k$, because otherwise the walk would not have reached edge $e$ in the first place. Hence, $e$ induces a split with the desired property. \end{proof}

\begin{lemma}
\label{lem:atleasttwo}
Let $\T$ be a tree on $n \geq 3k-2$ taxa. Then $g_k(\T) \geq 2$.
\end{lemma}
\begin{proof}
Let $A|B$ be the split whose existence is guaranteed by Lemma \ref{lem:nicecut}. Due to the fact that $k \leq |B| \leq 2k-2$, we have that $|A| \geq |X|-(2k-2) \geq (3k-2) - (2k-2) \geq k$. Hence, the character with two states, $A$ and $B$, is definitely a $g_k$ character. Combined with the fact that $X$ (i.e. the character with a single size-$n$ state) is vacuously also a $g_k$ character, we have that $g_k(\T) \geq 2$.
\end{proof}

We say that a tree $\T$ on $n$ vertices is \emph{fully $k$-loaded} if it can be created by the following process, which is schematically illustrated in Figure \ref{fig:scaffold}. First, select a tree $S$ on $\lceil \frac{n}{k-1} \rceil$ taxa, we call this the \emph{scaffold} tree. If $n \text{ mod } (k-1) = 0$, we replace each leaf of $S$ with a subtree on $k-1$ taxa. If $n \text{ mod } (k-1) \neq 0$, we replace all but one of the leaves of $S$ with a subtree on $k-1$ taxa, and the remaining leaf with a subtree on $n \text{ mod } (k-1)$ taxa. We call this remaining subtree, if it exists, the \emph{residue} subtree (and the corresponding leaf of the scaffold the \emph{residue leaf}). We note that the subtrees used in this process, do not have to have the same topology, and that at least one fully $k$-loaded tree exists for every $n$. We will write \emph{fully loaded tree} when $k$ is implicit from the context. 

\begin{figure}[h]
\centering
\includegraphics[scale=0.18]{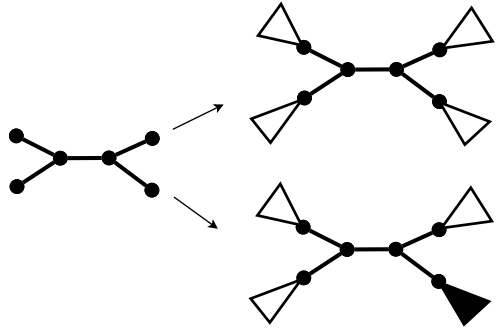}
\caption{The tree on the left is a scaffold tree with
4 vertices. If $n \text{ mod } (k-1) = 0$, then each of its 4 leaves is replaced by a subtree on $k-1$ taxa, depicted as white triangles in the top tree. If $n \text{ mod } (k-1) \neq 0$, then exactly one of these subtrees is allowed to have
$n \text{ mod } (k-1)$ taxa - this is the residue subtree - shown as a black triangle in the bottom tree. 
}
\label{fig:scaffold} 
\end{figure}

As an explicit example, consider the tree in Figure \ref{fig:g2}. This is a fully $4$-loaded tree, because you could obtain it from the (unique) scaffold tree on 3 leaves, by attaching subtrees with taxa $\{a,b,c\}$, $\{e,f,g\}$ and $\{d\}$ to the leaves of the scaffold. In this case, $\{d\}$ would be the residue subtree, because it does not contain exactly $4-1=3$ taxa. The tree is actually also fully $5$-loaded, by taking a single edge as the scaffold, and attaching subtrees on $\{a,b,c,d\}$ and $\{e,f,g\}$ (where here $\{e,f,g\}$ is the residue). However, the tree is not a fully $3$-loaded tree, because more than one of its taxa are not in cherries.


\begin{lemma}
\label{lem:allsame}
Let $\T$ be a fully $k$-loaded tree on $n$ taxa and $S$ be a corresponding scaffold tree. Then $g_k(\T) = g_2(S) = g_2(\lceil \frac{n}{k-1} \rceil)$. In particular: all fully $k$-loaded trees on $n$ taxa have the same $g_k$ value.
\end{lemma}
\begin{proof}
Recall that,
in the process of constructing $\T$ from $S$, at most one of the leaves of $S$ is replaced with a subtree with strictly fewer than $k-1$ taxa; all other leaves are replaced by subtrees with \emph{exactly} $k-1$ taxa. Now, each
state in a $g_2$ character of $S$ contains at least 2 leaves of $S$, so at least one of these leaves corresponds to a subtree of $\T$ with exactly $k-1$ taxa (and the other leaf corresponds to a subtree of $\T$ with at least 1 taxon). Hence,  when this $g_2$ character is mapped to a character of $\T$ each state in the character of $\T$ has at least $k$ taxa i.e. it is a $g_k$ character. 

In the other direction, consider an arbitrary $g_k$ character of $\T$. 
The subtrees of $\T$ corresponding to leaves of $S$ all have at most $k-1$ taxa, so any $g_k$ character of $\T$ must intersect with taxa from at least two such subtrees. When projected back to a character $S$ in the natural fashion, this induces a character of $S$ that has at least two leaves per state.

The second inequality follows from \cite{kelk2017note}, where it is established that $g_2$ depends only on $n$, and not the topology of the tree.
\end{proof}

A subtree of $\T$ is \emph{pendant} if it can be detached from $\T$ by deleting a single edge. The subtrees attached in the construction of fully loaded trees are all pendant, for example. The following lemma, which applies to all trees, shows that the topology of ``small' (with respect to $k$) pendant subtrees  is irrelevant for determining $g_k$.

\begin{lemma}
Let $\T$ be a tree on $n>k$ taxa. Suppose $\T$ contains a split $A|B$ where $|B| \leq k$. Let $\T'$ be the tree obtained by replacing the $B$ subtree with an arbitrary other subtree on the same subset of taxa. Then $g_k(\T) = g_k(\T')$.
\end{lemma}
\begin{proof}
This is a direct consequence of Observation \ref{obs:atleasttwo}.
\end{proof}

Our final auxiliary lemma shows that ``moderately large'' pendant subtrees (with respect to $k$) can be replaced by subtrees that, in a local sense, are fully $k$-loaded, \emph{without increasing $g_k$}.

\begin{lemma}
\label{lem:replace}
Let $\T$ be a tree on $n$ taxa. Suppose $\T$ contains a split $A|B$ such that $k \leq |B| \leq 2(k-1)$. Consider the unique fully $k$-loaded tree $\mathcal{F}$ on $|B|$ taxa obtained from the scaffold tree that is just a single edge $e$. Create a new tree $\T'$ on $n$ taxa by replacing the size-$|B|$ subtree of $\T$, with $\mathcal{F}$, where $\mathcal{F}$ is attached by subdividing $e$. Then $g_k(\T) \geq g_k(\T')$.
\end{lemma}
\begin{proof}
Note that, due to our choice of where to attach $\mathcal{F}$, the subtree $\mathcal{F}$ (when viewed as part of $\T$) consists of two subtrees: one containing exactly $k-1$ taxa, and the other containing at most $k-1$ taxa. Every $g_k$ character of $\T'$ must have a state that includes all the taxa from both subtrees: this is a consequence of applying Observation \ref{obs:atleasttwo} to both the subtrees. Any such character is also a $g_k$ character for $\T$.
\end{proof}

We now move towards our main result. The following lemma establishes the base case of the induction proof used in Theorem \ref{thm:lowerbound}.

\begin{lemma}
\label{lem:basecase}
For $k \leq n \leq 4k - 4$, there exists a fully $k$-loaded tree on $n$ vertices that minimizes $g_k$.
\end{lemma}
\begin{proof}
For $k \leq n \leq 2k-1$ this is immediate from Observation \ref{obs:gettingstarted} and the fact that fully $k$-loaded trees exist for every $n$. For $2k \leq n \leq 3k-3$, observe that any fully $k$-loaded tree $\T$ obtained from the unique scaffold tree on 3 leaves (i.e. the star tree with single degree-3 node incident to 3 leaves), has $g_k(\T) = 1$ which is minimal. For $3k - 2 \leq n \leq 4k-4$, observe that a fully $k$-loaded tree $\T$ obtained from the unique  scaffold tree on 4 vertices, has $g_k(\T)=2$. This must be minimum, due to the lower bound established by Lemma \ref{lem:atleasttwo}. 
\end{proof}

\begin{theorem}
\label{thm:lowerbound}
Let $n\geq k$. Every fully $k$-loaded tree $\T$ on $n$ taxa is a minimizer for $g_k$.
\end{theorem}
\begin{proof}
We prove this by induction. Our inductive claim is not that \emph{every} fully $k$-loaded tree on $n$ taxa is minimum, but that \emph{at least one} is. The strengthening to \emph{every} will then follow automatically from Lemma \ref{lem:allsame}. We assume $k \geq 3$, because for $k \in \{1,2\}$ the result is immediate due to topological neutrality.

For the base case we take $n \leq 4k-4$. This is proven by Lemma \ref{lem:basecase}.

Now, assume that for all $n \leq N$ the claim holds, where $N \geq 4k-4$. Let $\T$ be an arbitrary tree on $n = N+1$ taxa that minimizes $g_k$. From Lemma \ref{lem:nicecut}, there exists a split $E|F$ in $\T$ such that $k \leq |F| \leq 2(k-1)$. By applying Lemma \ref{lem:replace}, we obtain a new tree $\T'$ on the same number of taxa (that also contains the split $E|F$) such that $g_k(\T) \geq g_k(\T')$, so $\T'$ is also a minimizer -- because by assumption $\T$ was a minimizer. For ease of notation, we will take $\T = \T'$ in the rest of the proof.

Crucially, by construction, the $F$ subtree of $\T$ now consists of one subtree containing \emph{exactly} $k-1$ taxa, and one subtree containing at least 1 and \emph{at most} $k-1$ taxa. Hence, $\T$ contains a tripartition $A|B|C$ where $|B|=k-1$ and $1 \leq |C| \leq k-1$ and $|A| > 2(k-1) \geq k$ (where the last inequality follows because $n > N \geq 4k-4$ and $k \geq 3$). Let $a,b,c$ be the cardinalities of $A, B, C$ respectively.

The following recurrence holds\footnote{The recurrence could actually be simplified because $g_k(\T|C)=0$, by virtue of the fact that $c < k$, and $g_k(\T|B \cup C)=1$, because $k \leq b+c < 2k$. However, this simplification is not necessary for
the proof and obfuscates the origin of the recurrence.}, because $b=k-1$:
\begin{align}
g_k(\T) = g_k(\T|A \cup B)g_k(\T|C) +
g_k(\T|A)g_k(\T|B \cup C) +
g_k(\T|A \cup B)g_k(\T|B \cup C).
\end{align}
The three terms of the recurrence are based on Observation \ref{obs:atleasttwo} i.e. that every character we wish to count has a state that is a \emph{strict} superset of $B$. Such a state must also intersect with (i) $A$ (but not $C$), or (ii) $C$ (but not $A$) or (iii) both.
The third term, which models (iii), requires a little explanation. It is correct because every such state can be decomposed into a part that intersects both $B$ and $A$, and a part that intersects both $B$ and $C$; both these parts have size at least $k$. In the other direction, $g_k$ characters for $\T|A \cup B$ (respectively, $\T|B \cup C$) necessarily contain a state that contains all of $B$, and at least one taxon from $A$ (respectively, $C$), so every pairing of a $g_k$ character
from $\T|A \cup B$ with a $g_k$ character
from $\T|B \cup C$ yields a character of type (iii).


Note that $a,b,c, a+b,b+c$ are all strictly less than $n$, so the inductive claim holds for all trees restricted to subsets of taxa of these sizes.
Let $g_k^{fl}(n)$ denote the value of $g_k$ for a fully $k$-loaded tree on $n$ taxa; as proven earlier this is independent of topology. By induction we therefore have,
\begin{align}
\label{ineq:induc}
g_k(\T) \geq g_k^{fl}(a+b)g_k^{fl}(c) +  g_k^{fl}(a)g_k^{fl}(b+c) +
g_k^{fl}(a+b)g_k^{fl}(b+c).
\end{align}

The next part of the proof has two phases of applying the inductive hypothesis. If at least one of $a$ and $c$ is divisible by $k-1$, then only the first phase will be necessary.

We create a new tree $\T'$ on $n$ taxa by doing the following. We replace the $A$ subtree with an arbitrary fully $k$-loaded tree on $a$ taxa. The attachment point of this fully loaded tree should be a subdivided edge of its underlying scaffold tree. (This is necessary to avoid that we accidentally attach it ``inside'' one of its size-($k-1$) subtrees, which would destroy its fully $k$-loaded structure). Note that such a scaffold tree definitely contains at least one edge, because $a \geq k$.
Moreover, if $a$ is not divisible by $k-1$, we must use as attachment point a subdivision of the unique edge of the underlying scaffold tree that feeds into the residue leaf: in Figure \ref{fig:scaffold} this would mean the edge feeding into the black triangle. (The significance of this will be explained later).

We leave the $B$ and $C$ subtrees untouched. This completes the construction of $\T'$. 

Observe that the subtrees $\T'|A \cup B$, 
$\T'|C$,
$\T'|B \cup C$, $\T'|A$ are all fully $k$-loaded; for $\T'|B \cup C$ and $\T'|A \cup B$ this follows because $b=k-1$ and (for $\T'|A \cup B$) because of our choice of attachment point. $\T'|A$ is fully $k$-loaded by construction, and $\T'|C$ is vacuously fully $k$-loaded because $1 \leq c < k$.
Hence, $g_k(\T')$ is identical to the right hand side of inequality (\ref{ineq:induc}). Hence, $g_k(\T) \geq g_k(\T')$, so due to the fact that we assumed that $\T$ was a minimizer, $\T'$ is also a minimizer. Now, if at least one of $a$ and $c$ is divisible by $k-1$, it follows that $\T'$ is not only a minimizer, but also a fully $k$-loaded tree, so the proof is complete.

If neither $a$ nor $c$ is divisible by $k-1$, then $\T'$ will not be fully $k$-loaded. This is because the $A$ part of $\T'$ will contain a residue subtree, and the $C$ part of $\T'$ is (entirely)
a residue subtree, yielding two residue subtrees in total; a fully $k$-loaded tree has at most one residue subtree. In this situation, depicted in Figure \ref{fig:phasetwo}, a second phase of induction is necessary to ``clean up'' these unwanted subtrees. Observe that $a > 2k-2$, so the underlying scaffold tree for $A$ has at least 3 leaves (because a fully $k$-loaded tree with an underlying scaffold consisting of just 2 leaves, can have at most $2(k-1)$ taxa). Hence, the underlying scaffold tree contains a cherry, such that neither leaf of the cherry is a residue leaf. In more detail: if the scaffold has 4 or more leaves, it will have at least 2 cherries, and the residue leaf can be part of at most one of them. If the scaffold has 3 leaves, then any pairing of 2 leaves forms a cherry, so take the 2 non-residue leaves.

\begin{figure}[h]
\centering
\includegraphics[scale=0.2]{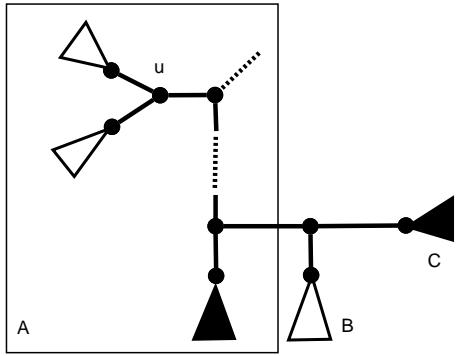}
\caption{If neither $|A|$ nor $|C|$ is divisible by $k-1$, then $\T'$ will not yet be fully loaded, due to having \emph{two} residue subtrees (shown as black triangles). In this case a second application of induction is required.
}
\label{fig:phasetwo} 
\end{figure}

Hence, both leaves of this cherry expand to subtrees with exactly $k-1$ taxa. By construction the attachment point explicitly avoids both these subtrees and the two edges feeding into them; this is the significance of the earlier careful choice of attachment location in the case that $a$ is not divisible by $k-1$. As a result, these two subtrees with $k-1$ taxa have a common parent in $\T'$ (this is the node $u$ in Figure \ref{fig:phasetwo}). This means that $\T'$ contains a tripartition $A'|B'|C'$ where $|B'|=|C'|=k-1$, by taking these two size $k-1$ subtrees as $B'$ and $C'$. We repeat the original induction argument with this new tripartition, yielding a new tree $\T''$ with the property that $g_k(\T') \geq g_k(\T'')$. Due to the fact that (at least) two parts of the new tripartition are divisible by $k-1$, $\T''$ will be fully $k$-loaded (i.e. we will only need the first phase of the proof), and a minimizer, so the inductive claim follows.
\end{proof}

\begin{corollary}
For $n \geq k \geq 2$, the minimum value of $g_k$ ranging over all trees on $n$ taxa is exactly
\[
\bigg \lfloor \frac{\phi^{\lceil \frac{n}{k-1} \rceil-1}}{\sqrt{5}} + \frac{1}{2} \bigg \rfloor.
\]

\end{corollary}
\begin{proof}
Recall from the preliminaries that an exact expression for $g_2$ is
\begin{align*}
g_2(n) &= \bigg \lfloor \frac{\phi^{n-1}}{\sqrt{5}} + \frac{1}{2} \bigg \rfloor.\\
\end{align*}
The desired new expression is obtained by substituting $\lceil \frac{n}{k-1} \rceil$ for $n$ in the $g_2$ expression (see Lemma \ref{lem:allsame}), while Theorem \ref{thm:lowerbound} establishes that this is indeed minimum.
\end{proof}

Returning to our concrete example of $g_3$, this yields a minimum of $\Theta(\phi^{n/2})$, so $\Theta(1.272^n)$, contrasting with the
caterpillar-induced maximum of $\Theta(1.466^n)$. Given the gap between the maximum and minimum,
it is reasonable to conjecture that $g_3$ drops smoothly as the number of cherries $t$ in the
tree increases (because fully $3$-loaded trees consist primarily of cherries, with at most 1 taxon not in a cherry). The following bound shows that this is \emph{partially} the case.

\begin{lemma}
Let $\T$ be a tree on $n$ taxa with $t$ cherries. Then $g_3(\T)$ is $O( \phi^{n-t})$.
\end{lemma}
\begin{proof}
Let $\T'$ be the fully $3$-loaded tree obtained from $\T$ by doubling the $n-2t$ taxa not in cherries. $\T'$ has $n+ (n-2)t = 2n-2t$ taxa. This is even, so from Lemma \ref{lem:allsame} we have that,
\begin{align*}
g_3(\T') &= \bigg \lfloor \frac{\phi^{n-t-1}}{\sqrt{5}} + \frac{1}{2} \bigg \rfloor.
\end{align*}
Observe that every $g_3$ character of $\T$ naturally maps injectively to a $g_3$ character of $\T'$, by putting each newly added taxon into the same state as its sibling. Hence $g_3(\T) \leq g_3(\T')$ and the result follows.
\end{proof}

Although the upper bound given above does converge on the true minimum (consider $t=n/2$), it remains quite weak, in the sense that once $t$ falls below (roughly) $n/5$, the caterpillar upper bound of $O(1.466^n)$ is stronger. It would be interesting to develop tighter upper bounds when $t$ is comparatively small compared to $n$.


\section{Applications}
\label{sec:app}
We note in the table below minimum and maximum values of $g_k$, for the first few values of $k$. Each entry  $\alpha$ indicates that the minimum (respectively,  maximum) grows at rate $\Theta(\alpha^n)$. For $k \geq 3$ we have used the results in this article: fully loaded $k$-trees determine the minimum and caterpillars determine the maximum.

\begin{table}[h]
\begin{tabular}{|c|c|c|}
\hline
     & \emph{minimum} & \emph{maximum} \\ \hline
$g_1$ & 2.618            &  2.618           \\ \hline
$g_2$ & 1.618            &  1.618           \\ \hline
$g_3$ & 1.272            &       1.466      \\ \hline
$g_4$ & 1.174            &     1.380       \\ \hline
$g_5$ & 1.128           &       1.325      \\ \hline
$g_6$ & 1.101          &         1.285    \\ \hline
\end{tabular}
\end{table}

The numbers in the right column, let us denote these $\alpha_k$, have an algorithmic significance.  Any algorithm that functions by looping through all the $g_k$ characters, and performing polynomial-time computation on each such character, will have a worst-case running time of the form
$O(\alpha_k^n \cdot \text{poly}(n))$. This is a consequence of the results from \cite{kelk2017note}, where it is shown how to compute $g_k$ in polynomial time, and (relatedly) how to list all $g_k$ characters in time $\Theta(g_k(\T) \cdot \text{poly}(n))$. Clearly, this worst-case running time will decrease as $k$ increases. 

This raises the question of how scalable such algorithms can be in practice.  To answer this, we took the code from \cite{kelk2017note} and for $1 \leq k \leq 6$ used it to determine the highest $n$ for which all $g_k$ characters could be listed within 1, 10 or 100 seconds. We did this twice: once on caterpillar trees, and once on randomly generated trees (which are typically far from being caterpillars due to being fairly balanced). The random trees were generated following the protocol described in \cite{van2020reflections}. The results are below\footnote{The experiment was conducted within
the Linux Subsystem (Ubuntu 16.04.6 LTS), running under Windows 10, on a 64-bit HP Envy Laptop 13-ad0xx (quad-core i7-7500 @ 2.7 GHz), with 8 Gb of memory.}.


\begin{table}[h]
\begin{tabular}{|c|c|c|c|c|c|c|}
\hline
   & \multicolumn{3}{c|}{\emph{caterpillar}} & \multicolumn{3}{c|}{\emph{random trees}} \\ \hline
   & 1s          & 10s          & 100s         & 1s          & 10s          & 100s         \\ \hline
\hline
$g_1$ &    14         &    16          &    19          &   14          &   16           &      19        \\ \hline
$g_2$ &     27        &     32         &   37           &    27         &   32           &  37            \\ \hline
$g_3$ &     34        &    41          &    47          &     38        &   49           &   55           \\ \hline
$g_4$ &     40        &    48          &    56          &   56          &    66          &   74           \\ \hline
$g_5$ &     47        &     56         &   64           &   73          &    84          &    96          \\ \hline
$g_6$ &     52        &     63         &    72          &     83        &   101           &    116          \\ \hline
\end{tabular}
\end{table}

We see that, for $k\geq 3$, the listing algorithms scale somewhat better on random trees than on caterpillars. Beyond extremely small $n$, real phylogenetic trees inferred from biological data rarely reach either topological extreme, so for such trees it is realistic to expect an intermediate level of scalability. 

Although the running times remain prohibitive for larger trees, they are still quite encouraging by the standards of exponential-time algorithms. This was noted in \cite{kelk2017note} where an algorithm based on listing $g_2$ characters was used to design a simple but surprisingly practical algorithm for exact computation of \emph{maximum parsimony distance} \cite{fischer2014maximum} on two phylogenetic trees. This later became the foundation for a far more scalable sampling-based heuristic for the same problem \cite{van2020reflections}; both the listing and sampling leverage the dynamic programming scaffolding originally used to actually count $g_k$. In \cite{kelk2017note} it was also observed that the \emph{maximum agreement forest} problem \cite{AllenSteel2001} can be solved in time $O(2.619^n)$ by listing $g_1$ characters. One can think of an agreement forest on a set of trees as a character that is convex on all trees and where the states induce the same topologies across all trees. There are many convex characters that do not induce agreement forests, but - crucially - every agreement forest \emph{does} induce a convex character. 

Indeed, any optimization or decision problem that `projects down' onto convex characters, can potentially be tackled by such listing algorithms. The flexibility of such an approach is that one only requires an algorithm to check whether a convex character actually has the desired, typically stronger, property, and thus little time needs to be spent on algorithm design or research. The new upper bounds given in this article additionally allow us to bound the worst-case running time of such algorithms. Some simple examples of problems that could easily be modelled within this `convex character programming' framework are listed below. In all cases we assume the input is a set of phylogenetic trees $\mathbb{T}$ on $X$ whereby $|\mathbb{T}|$ is $O(\text{poly}(n))$. The notation $O^{*}(.)$ denotes that polynomial factors are suppressed. 
\begin{itemize}
    \item \emph{Partition problems with lower bounds:} Compute an agreement forest for $\mathbb{T}$ with a minimum number of components, such that every component of the forest contains at least $k$ taxa, or state that no such forest exists. Here $g_k$ can be used, yielding a running time of $O^{*}( \alpha_k^n)$. Generalizes easily to the variant where disjointness is only required in one or more of the input trees.
    \item \emph{Exact partition problems: } Determine whether  it is possible to perfectly partition $X$ into size-4 blocks, such that in each tree in $\mathbb{T}$ the induced size-4 subtrees (``quartets'') are disjoint and have the same topology. Here $g_4$ can be used, yielding a running time of $O^{*}(1.38^n)$.
    \item \emph{Character inference problems subject to an objective function: } Let $\T$ be a specified tree in the input set $\mathbb{T}$. Find a $g_k$ character $f$ on $\T$ that optimizes a polynomial-time computable objective function $Z(f,\mathbb{T})$, such as sum of parsimony scores, or sum of log likelihoods subject to a chosen statistical model of evolution. Here $g_k$ can be used, yielding a running time of $O^{*}( \alpha_k^n)$. Can be generalized to the case when $\T \in \mathbb{T}$ is not given and the choice of tree from $\mathbb{T}$ needs to be co-optimized.
\end{itemize}

\section{Future work}
\label{sec:future}
It would be particularly interesting to quantify the growth of $g_k$ between the two topological extremes of fully $k$-loaded trees and caterpillars. For $g_3$, for example, it is clear that the value drops towards its minimum as the number of cherries in a tree increases, but as discussed earlier it is only a partial explanation for the behaviour of $g_3$. Thus, it is necessary to identify which topological features of trees, if any, are sufficient (and possibly necessary) to yield a decrease in $g_k$. It would also be interesting to dive deeper into a question posed in \cite{kelk2017note}: what are necessary and sufficient conditions for two non-isomorphic trees on $n$ taxa to induce the same vector of $g_k$ values?
How might such a vector of $g_k$ values change under the action of \emph{tree rearrangement} operations? These are commonly used to heuristically navigate through the space of trees in the process of inferring phylogenetic trees from sets of characters \cite{john2017shape}; see \cite{bryant2004splits} for similar questions.



\section{Acknowledgements}

Ruben Meuwese was supported by the Dutch Research Council (NWO) KLEIN 1 grant \emph{Deep kernelization for phylogenetic discordance}, project number OCENW.KLEIN.305.

\bibliographystyle{plain}
\bibliography{Convex2021}

\end{document}